\documentclass[a4paper,10pt]{article}
\usepackage[utf8]{inputenc}

\usepackage{ifpdf}

\usepackage{amsmath}
\usepackage{amsthm}
\usepackage{amssymb}

\usepackage{relsize}

\usepackage[T1]{fontenc} 

\usepackage{mathrsfs}
\usepackage{eucal}

\usepackage[ugly]{units} 

\usepackage[affil-it]{authblk} 

\usepackage{wrapfig}

\ifpdf
	\usepackage[pdftex]{graphicx}
	\usepackage{epstopdf}
\else
	\usepackage{graphicx}
\fi

\usepackage{hyperref}

\newtheorem{proposition}{Proposition}[section]
\newtheorem{corollary}{Corollary}[section]

\theoremstyle{definition}

\newtheorem{remark}{Remark}[section]

\newcommand{\set}[1]{\mathsf{#1}}
\newcommand{\ud}{\, \mathrm{d} }

\renewcommand{\i}{\mathrm{i}}

\newcommand{\realset}{\mathbb{R}}

\newcommand{\C}{\mathcal{C}}

\newcommand{\vect}[1]{\boldsymbol{#1}}

\newcommand{\vx}{\vect{x}}

\newcommand{\manifold}[1]{\mathscr{#1}} 

\newcommand{\PP}{\manifold{P}}
\renewcommand{\P}{\PP}

\newcommand{\Xcal}{\mathfrak{X}}

\newcommand{\Diff}{\mathrm{Diff}}

\newcommand{\LieD}[1]{\mathcal{L}_{\! #1}}
\newcommand{\flow}{\varphi}
\newcommand{\dflow}{\Phi}

\providecommand{\spg}{\mathfrak{sp}}
\providecommand{\GL}{\mathrm{GL}}
\providecommand{\gl}{\mathfrak{gl}}
\providecommand{\Galg}{\mathfrak{g}}
\providecommand{\Ggroup}{\set{G}}

\newcommand{\onSKF}[2]{#2}

\title{Geometric Integration of Non-autonomous Systems with Application to Rotor Dynamics}
\author[1,2]{Klas MODIN}
\onSKF{
	\affil{
	$^a$SKF Engineering \& Research Centre \\
	MDC, RKs--2, SE--415 50 Göteborg, Sweden
	}
	\affil{
	$^b$Centre for Mathematical Sciences, Lund University \\
	Box 118, SE--221 00 Lund, Sweden \\
	E--mail: {\selectfont\ttfamily\itshape kmodin@maths.lth.se}
	}
}{
	\affil[1]{
	Centre for Mathematical Sciences, Lund University \\
	Box 118, SE--221 00 Lund, Sweden \\
	E--mail: {\selectfont\ttfamily\itshape kmodin@maths.lth.se}
	}
	\affil[2]{
	SKF Engineering \& Research Centre \\
	MDC, RKs--2, SE--415 50 Göteborg, Sweden
	}
}

\date{March 25, 2009}

\begin{document}

\maketitle


\newcommand{\Pext}{\overline{\set{P}}}

\begin{abstract}
	Geometric integration of non-autonomous classical engineering problems,
	such as rotor dynamics, is investigated. It is shown, both numerically and by
	backward error analysis, that geometric (structure preserving) integration
	algorithms are superior to conventional Runge--Kutta methods.  
	\par\smallskip
	{\bf Key-words:} Geometric numerical integration, splitting methods, rotor dynamics
\end{abstract}


\section{Introduction} 
\label{sec:Introduction}

In this paper we study geometric numerical integration algorithms for 
non-autonomous systems. By classifying the appropriate Lie sub-algebra
of vector fields, the standard framework for backward error analysis
can be used to explain the superior qualitative behaviour of
geometric methods based on the splitting approach.

The current section continues with a brief review of 
the general framework for geometric methods, mainly following
the approach by Reich~\cite{Re1999}.
In Section~\ref{sec:linear_systems} we study geometric integration
of linear systems with non-constant periodic coefficients.
A numerical example from classical rotor dynamics is given.
%
Conclusions are given in Section~\ref{sec:conclusions}.

We adopt the following notation. $\P$ denotes a phase space manifold of dimension
$n$, with local coordinates $\vx=(x_1,\ldots,x_n)$. In the case when $\P$
is a linear space we also use $\set{P}$.
Further, $\Xcal(\P)$ denotes the
linear space of vector fields on~$\P$.
The flow of $X\in\Xcal(\P)$ is denoted $\flow^t_X$, where
$t$ is the time parameter.
The Lie derivative
along $X$ is denoted~$\LieD{X}$.
If~$X,Y\in\Xcal(\P)$ then the vector field commutator~$[X,Y]_\Xcal=\LieD{X}Y$
supplies~$\Xcal(\P)$ with an infinite dimensional Lie algebra structure.
Its corresponding Lie group is
the set~$\Diff(\P)$ of diffeomorphisms on~$\P$, with composition
as group operation.

\begin{remark}
	More precisely, the group $\Diff(\P)$ has the structure
	of a \emph{Fréchet Lie group}. See~\cite{Ha1982,Sc2004} for
	issues concerning infinite dimensional Lie groups.
\end{remark}


As usual, the general linear group of $n\times n$--matrices is denoted
$\GL(n)$ and its corresponding Lie algebra
$\gl(n)$. We use $[A,B]_\gl$ for the matrix commutator~$AB-BA$. 


If $\set{V}$ is a metric linear space, then the linear space of 
smooth periodic functions $\realset\to\set{V}$
with period $2\pi/\Omega$ is denoted $\C_\Omega(\set{V})$. Notice that
this space is closed under differentiation, i.e., if $f \in \C_\Omega(\set{V})$
then it also holds that $f' \in \C_\Omega(\set{V})$.

\subsection{Geometric Integration and Backward Error Analysis} 
\label{sub:geometric_integration_and_backward_error_analysis}

Let $\Xcal_S(\P)$ be a sub-algebra of $\Xcal(\P)$, i.e.,
a linear sub-space which is closed under the commutator.
Its corresponding sub-group of $\Diff(\P)$ is denoted
$\Diff_S(\P)$.
Let $X\in\Xcal_S(\P)$ be a vector field which is to be integrated
numerically. Assume that $X$ can be splitted as a sum
of explicitly integrable vector field also belonging to
$\Xcal_S(\P)$. That is, 
$X=Y+Z$ where $Y,Z \in\Xcal_S(\P)$ and
$\flow^t_{Y},\flow^t_{Z}$ can be computed explicitly.
By various compositions, various numerical integration
schemes for $\flow^t_X$ are obtained. 
The most classical example is $\dflow_h = \flow_Y^{h/2}\circ
\flow_Z^{h\vphantom{/2}}\circ\flow_Y^{h/2}$, which yields a second
order symmetric method ($h$~is the step-size parameter of the method).
Since $\flow^t_Y,\flow^t_Z\in \Diff_S(\P)$, and since $\Diff_S(\P)$
is closed under composition (since it is the group operation),
it holds that $\dflow_h \in \Diff_S(\P)$.
Thus, the splitting approach yields structure preserving methods,
which is a key property.

Backward error analysis for structure preserving integrators
deals with the question of finding a modified vector field $\tilde{X}\in\Xcal_S(\P)$
such that $\dflow_h = \flow_{\tilde{X}}^h$. In conjunction with perturbation
theory, such an analysis can be used to study the dynamical properties of~$\dflow_h$. 
For splitting methods, backward error analysis is particularly simple as
the modified vector field, at least formally, 
is obtained from the Baker--Campbell--Hausdorff (BCH) formula.
For details on this framework we refer to Reich~\cite{Re1999}.

\section{Linear Systems} 
\label{sec:linear_systems}

In this section we study non-autonomous systems on a linear phase space $\set{P}$
with global coordinates $\vx=(x_1,\ldots,x_n)$. More precisely, let
$\Ggroup$ be a Lie sub-group of $\GL(n)$ and $\Galg$ its corresponding
Lie sub-algebra. We consider systems of the form
\begin{equation} \label{eq:linear_systems}
	\dot{\vx} = A(t)\vx + f(t)
\end{equation}
where $A \in \C_\Omega(\Galg)$ and $f \in \C_\Omega(\set{P})$ is a smooth vector valued
periodic function with period~$T=2\pi/\Omega$. Our objective is to construct
geometric integrators for~\eqref{eq:linear_systems}. Of course,
since the system is linear, there is a closed form formula
for its solution. However, in engineering applications, e.g.\ finite element analysis, 
the system is typically very large so computing the exponential matrix, which is
necessary for the exact solution, is not computationally efficient.
Also, it might not be possible
to analytically integrate $f$~and~$A$ over~$t$,
which is necessary for the exact solution.

In order to study dynamical systems of the form~\eqref{eq:linear_systems}
in the framework of geometric integration, we need, first of all, to 
extend the phase space to $\Pext=\set{P}\times\realset$ to include the time
variable in the dynamics. Coordinates on $\Pext$ are now
given by $(\vx,t)$ and the new independent variable is denoted~$\tau$
(in practice we always have $t(\tau)=\tau$).
Further, we need to find a Lie sub-algebra of $\Xcal(\Pext)$
which captures the form~\eqref{eq:linear_systems}. For this purpose, consider the
set of vector field on $\Pext$ given by
\begin{equation} \label{eq:linear_sub_algebra}
	\mathfrak{L}_{\Omega}(\set{P},\Galg) = \big\{ 
		X \in \Xcal(\Pext) \; \big| \; 
		X(\vx,t) = (A(t) x + f(t),\alpha), \; A\in\C_{\Omega}(\Galg),\; f\in\C_\Omega(\set{P}),\;\alpha\in\realset
	\big\} \; .
\end{equation}

We now continue with some results concerning properties of~$\mathfrak{L}_{\Omega}(\set{P},\Galg)$.
The first result states that it actually \emph{is} a Lie sub-algebra.

\begin{proposition} \label{prop:linear_lie_algebra}
	The set of vector fields $\mathfrak{L}_{\Omega}(\set{P},\Galg)$ is a Lie sub-algebra
	of $\Xcal(\Pext)$.
\end{proposition}

\begin{proof}
	We need to check that $\mathfrak{L}_{\Omega}(\set{P},\Galg)$ is closed under vector operations and under
	the Lie bracket. That is, $X,Y\in \mathfrak{L}_{\Omega}(\set{P},\Galg)$ should imply $a X + b Y \in \mathfrak{L}_{\Omega}(\set{P},\Galg)$ for
	$a,b\in\realset$ and $[X,Y]_\Xcal \in \mathfrak{L}_{\Omega}(\set{P},\Galg)$. 
	
	With $X(\vx,t)=(A(t)\vx+f(t),\alpha)$ and
	$Y(\vx,t)=(B(t)\vx+g(t),\beta)$ we get $(aX+bY)(\vx,t)=((a A+b B)(t)\vx+(a f+b g)(t),a \alpha+b \beta)$
	which is of the desired form. Further,
	\begin{multline*}
		[X,Y]_{\Xcal} = \begin{pmatrix} A(t) & A'(t)\vx+f'(t) \\ 0 & 0 \end{pmatrix} 
		\begin{pmatrix} B(t)\vx + g(t) \\ \beta
		\end{pmatrix}
		- 
		\begin{pmatrix} B(t) & B'(t)\vx+g'(t) \\ 0 & 0 \end{pmatrix} 
		\begin{pmatrix} A\vx + f(t) \\ \alpha
		\end{pmatrix}
		\\
		= \begin{pmatrix}
			(A(t)B(t)-B(t)A(t)+\beta A'(t)-\alpha B'(t))\vx + A(t)g(t) - B(t)f(t) + \beta f'(t) - \alpha g'(t) \\ 0
		\end{pmatrix}
	\end{multline*}
	which is of the desired form since $AB-BA+\beta A'-\alpha B'=
	[A,B]_\GL+\beta A' + \alpha B' \in \C_{\Omega}(\Galg)$ and 
	$(Ag - Bf + \beta f' - \alpha g') \in \C_\Omega(\set{P})$.
\end{proof}

From the proof above we immediately obtain the following corollary.

\begin{corollary}
	The set $\mathfrak{l}_\Omega=\C_{\Omega}(\Galg)\times\C_\Omega(\set{P})\times\realset$ equipped 
	with the induced vector operation
	\[
		a(A,f,\alpha)+b(B,g,\beta)=(aA+bB,af+bg,a\alpha+b\beta) , \quad a,b\in\realset
	\]
	and with the bracket operation
	\[
		[(A,f,\alpha),(B,g,\beta)]_{\mathfrak{L}} = ([A,B]_\GL+\beta A' + \alpha B',Ag-Bf+\beta f'-\alpha g',0)
	\]
	is a Lie algebra which is isomorphic to $\mathfrak{L}_{\Omega}(\set{P},\Galg)$ with isomorphism
	$\mathfrak{l}_\Omega \ni (A,f,\alpha) \mapsto (A\vx+f,\alpha) \in \mathfrak{L}_{\Omega}(\set{P},\Galg)$.
\end{corollary}

Since $\C_\Omega(\set{P})$ and $\C_{\Omega}(\Galg)$ are infinite dimensional 
it follows that $\mathfrak{l}_\Omega$, and therefore
also $\mathfrak{L}_{\Omega}(\set{P},\Galg)$, is infinite dimensional. However, a finite dimensional sub-space of 
$\C_\Omega(\set{P})$ is given by
\begin{equation} \label{eq:finite_frequency}
	\C_{\Omega,k}(\set{P}) = \big\{ f \in \C_\Omega(\set{P}) \;\big| \; 
	f(t) = \vect{a}_0 + \sum_{i=1}^k \vect{a}_i \cos(i \Omega t) + 
	\vect{b}_i \sin( i \Omega t), \; \vect{a}_i,\vect{b}_i\in\set{P} \big\}
\end{equation}
which is the sub-space of $\C_\Omega(\set{P})$ with angular frequencies bounded by~$k\Omega$.
Notice that the dimension of $\C_{\Omega,k}(\set{P})$ is $(2k+1)n$ and that 
$\C_{\Omega,\infty}(\set{P})=\C_{\Omega}(\set{P})$ and $\C_{\Omega,0}(\set{P})=\set{P}$. 
Further, $\C_{\Omega,k}(\set{P})$
is closed under differentiation. Clearly, these results also holds for
the corresponding sub-space $\C_{\Omega,l}(\Galg)$ of $\C_{\Omega}(\Galg)$,
except that the dimension is given by $(2l+1)\dim{\Galg}$ instead.

By replacing $\C_{\Omega}(\set{P})$ 
with $\C_{\Omega,k}(\set{P})$ and $\C_{\Omega}(\Galg)$ with $\C_{\Omega,l}(\Galg)$ 
we get the sub-spaces $\mathfrak{l}_{\Omega,k,l} = \C_{\Omega,l}(\Galg)\times\C_{\Omega,m}(\set{P})\times\realset$
of~$\mathfrak{l}_\Omega$.
In general $\mathfrak{l}_{\Omega,k,l}$ is \emph{not} a sub-algebra, due to the fact that
$A,B\in\C_{\Omega,l}(\Galg)$ does not in general imply $AB\in\C_{\Omega,l}(\Galg)$. However,
in some special cases the implication holds true.

\begin{proposition} \label{prop:finite_lie_algebra}
	The sub-spaces $\mathfrak{l}_{\Omega,k,0} = \Galg\times\C_{\Omega,k}(\set{P})\times\realset$
	and $\mathfrak{l}_{\Omega,k,\infty} = \C_{\Omega}(\Galg)\times\C_{\Omega,k}(\set{P})\times\realset$ are
	Lie sub-algebras of~$\mathfrak{l}_\Omega$. Further, $\mathfrak{l}_{\Omega,k,0}$ is finite dimensional
	with dimension $\dim\Galg + (2k+1)n + 1$.
\end{proposition}
Clearly, $\mathfrak{l}_{\Omega,k,0}$ and $\mathfrak{l}_{\Omega,k,\infty}$ induces corresponding Lie 
sub-algebras~$\mathfrak{L}_{\Omega,k,0}(\set{P},\Galg)$ and~$\mathfrak{L}_{\Omega,k,\infty}(\set{P},\Galg)$ of~$\mathfrak{L}_{\Omega}(\set{P},\Galg)$.

\subsection{Geometric Integration} 
\label{sub:geometric_integration}

In this section we describe an approach for geometric integration of
of systems of the form~\eqref{eq:linear_systems}. The approach is based on splitting.
To this extent we write~\eqref{eq:linear_systems} as an extended system
\begin{equation} \label{eq:linear_system_extended}
	\frac{\ud}{\ud\tau} \begin{pmatrix} \vx \\ t \end{pmatrix} =
	\begin{pmatrix} A(t)\vx + f(t) \\ 1 \end{pmatrix} \equiv X(\vx,t) \; .
\end{equation}
It is clear that $X\in\mathfrak{L}_{\Omega}(\set{P},\Galg)$. Since $\mathfrak{L}_{\Omega}(\set{P},\Galg)$ is a Lie sub-algebra
of $\Xcal(\Pext)$ it corresponds to a Lie sub-group
$\Diff_\Omega$ of $\Diff(\Pext)$. Geometric integration 
of~\eqref{eq:linear_system_extended} now means to find a one-step integration
algorithm~$\Phi_h \in \Diff_\Omega$.
There are of several ways to obtain such integrators.
One of the simplest, but yet most powerful ways, is to use a splitting
approach. That is, to split the vector field~$X$ as a sum of
two vector fields each of them of the form~\eqref{eq:linear_system_extended}.


%


\newpage

\subsection{Example: Linear Rotor Dynamics} 
\label{sub:example_linear_rotor_dynamics}

%
%
\begin{wrapfigure}[12]{r}[0ex]{0.25\textwidth} 
	\centering
	\raisebox{-29ex}[0ex][0ex]{\includegraphics[width=0.17\textwidth]{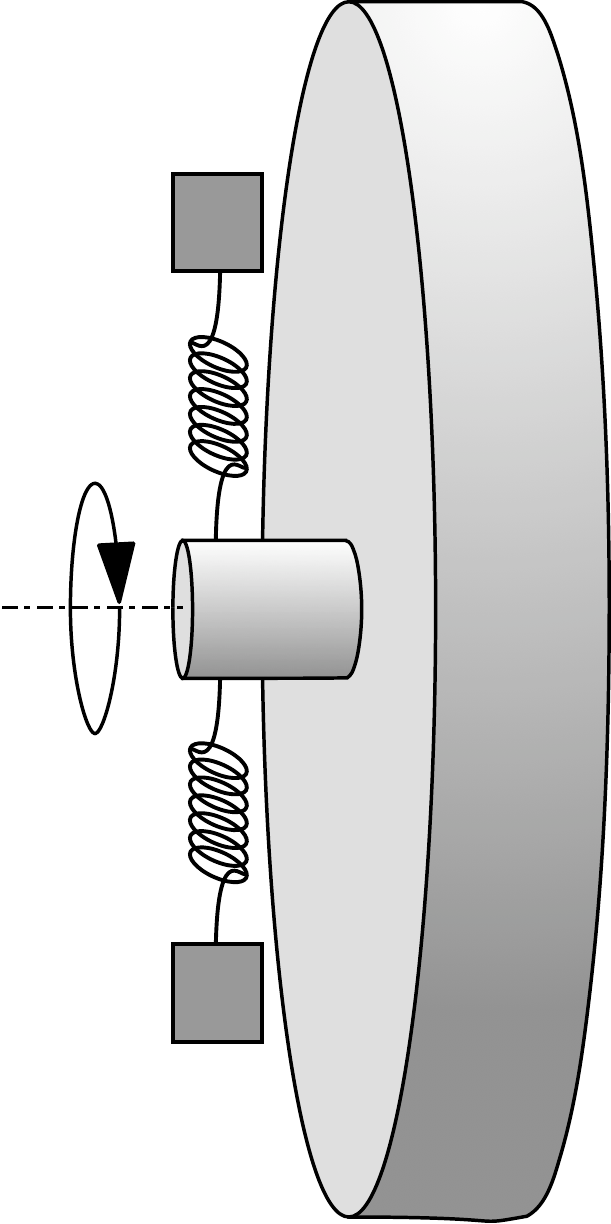}} 
\end{wrapfigure}

This example is the simplest possible rotor dynamical problem. It consists of
a disc attached to a shaft which is rotating with constant angular velocity~$\Omega$.
The shaft is held by a bearing, which is modelled as a linear spring with
stiffness~$k$. (See figure.)
The disc is slightly unbalanced, i.e., its centre of mass does not align with rotational
axis. This implies a time-dependent periodic centrifugal force acting on the rotor. 

The phase space for this system is given by $\set{P}=\realset^4$,
with coordinates~$\vx=(q_1,q_2,p_1,p_2)$, which is the horizontal and vertical position
of the shaft in a plane perpendicular to the axis of rotation, and their corresponding momenta.
The equations of motion are of the form~\eqref{eq:linear_system_extended} with
\begin{equation*} \label{eq:gov_simple_rotor}
	A = \begin{pmatrix}
		 0& 0 & m^{-1}& 0\\
		 0&0  & 0& m^{-1} \\
		 -k &0 & 0& 0\\
	0	& -k &0& 0
	\end{pmatrix}
	\qquad \text{and} \qquad
	f(t) = \varepsilon \Omega^2 \begin{pmatrix}
		0 \\ 0 \\ -\cos(\Omega t) \\ \sin(\Omega t)
	\end{pmatrix}
\end{equation*}
where $m$ is the total mass and $\varepsilon$ is the magnitude of the unbalance.

It holds that $A^T J + J A = 0$, so $A$ is an element in the canonical symplectic
Lie sub-algebra of~$\gl(4)$, i.e., we have~$\Galg=\spg(4)$.
Further, since $A$ is independent of $t$, and $f$ only contains a single
frequency, the appropriate Lie sub-algebra of~$\Xcal(\realset^4)$ is 
$\mathfrak{L}_{\Omega,1,0}(\realset^4,\spg(4))$, which is finite dimensional.

The eigenvalues of $A$ are $\pm \i \sqrt{k/m}$. Thus if $\Omega$ is close to
a multiple of the eigen frequency $\omega=\sqrt{k/m}$ of the system
starts to resonate. In this example we investigate how well various numerical
integrators capture that behaviour, both qualitatively and quantitatively.

For the data given in Table~\ref{tab:data} the problem is
numerically integrated with four different methods, two which 
are geometric and two which are not.

\begin{center}
	\begin{tabular}{lc}
	\textbf{Method} & \textbf{Geometric?} \\
	\hline\hline
	Implicit midpoint &  Yes \\
	Splitting method & Yes \\
	Heun's method & No \\
	Implicit extrapolation method & No
	\end{tabular}	
\end{center}

\begin{table} \label{tab:data}
	\centering
	\begin{tabular}{l|rl}
	 & \textbf{Value} & \textbf{Unit} \\
	\hline\hline
	$m$ & 1 & \unit{kg} \\
	$k$ & 1 & \unit{N/m} \\
	$\Omega$ & 1.02 & \unit{rad/s} \\
	$\varepsilon$ & 0.1 & $\unit{m \cdot kg}$ \\
	$\vx_0$ & (0,0,0,0) & \unit{(m,m,m/s,m/s)} 
	\end{tabular}
	\caption{Data used in the simulations of the rotor dynamical problem.}
\end{table}

The results of the $x_1$--variable are shown in Figure~\ref{fig:SimpleRotor_xplot}.
Notice that the geometric integrators captures the resonance phenomena in a qualitatively
correct way, whereas the non-geometric methods does not show the correct behaviour.

\begin{figure}
	\centering
		\includegraphics[width=0.99\textwidth]{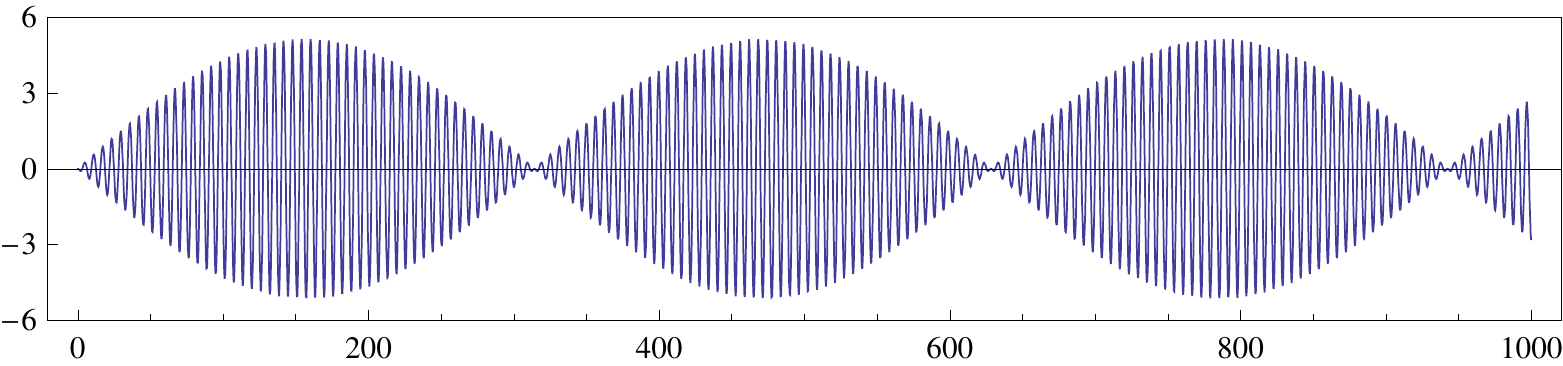} \\
		\includegraphics[width=0.99\textwidth]{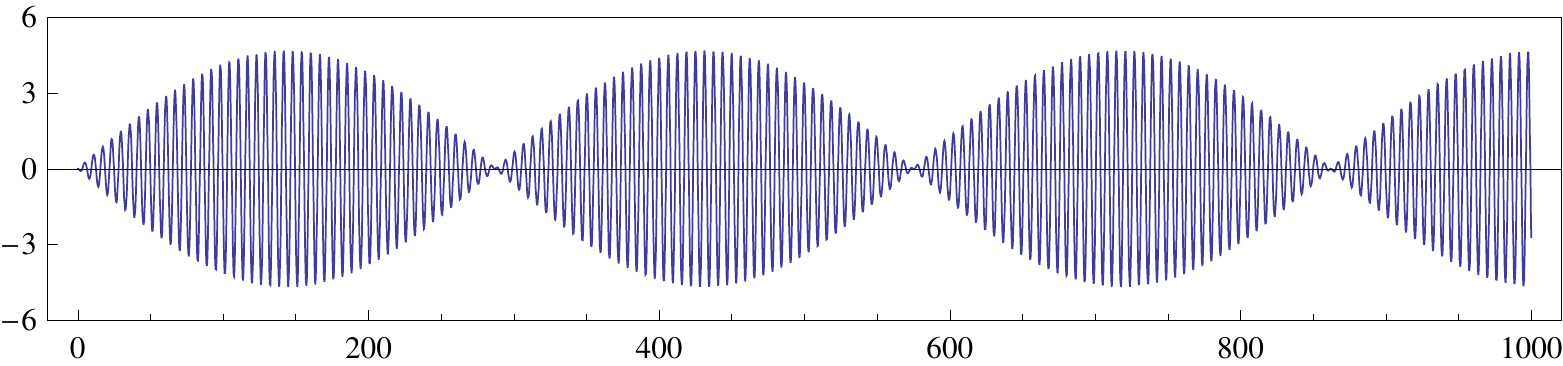} \\
		\includegraphics[width=0.99\textwidth]{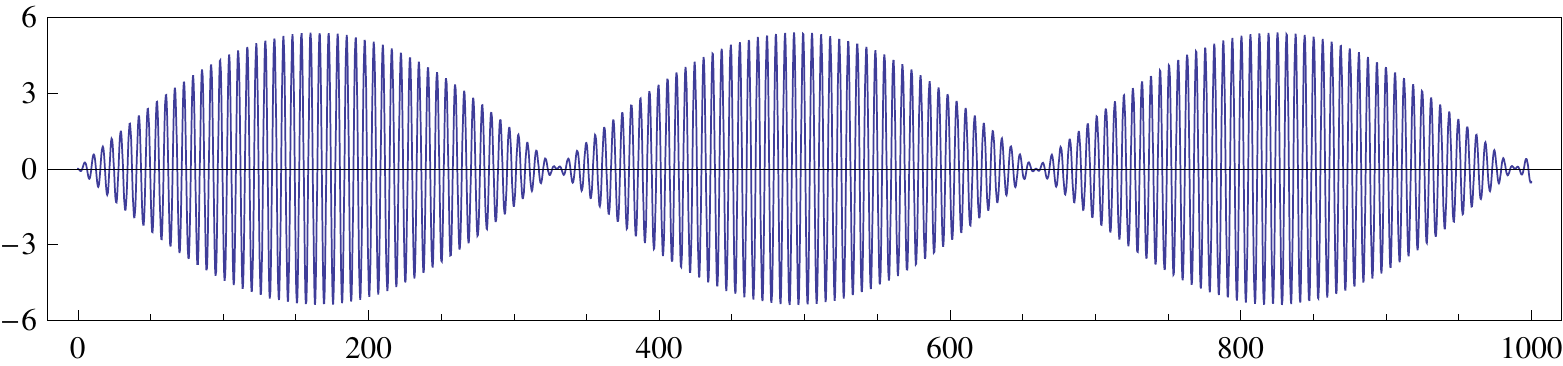} \\
		\includegraphics[width=0.99\textwidth]{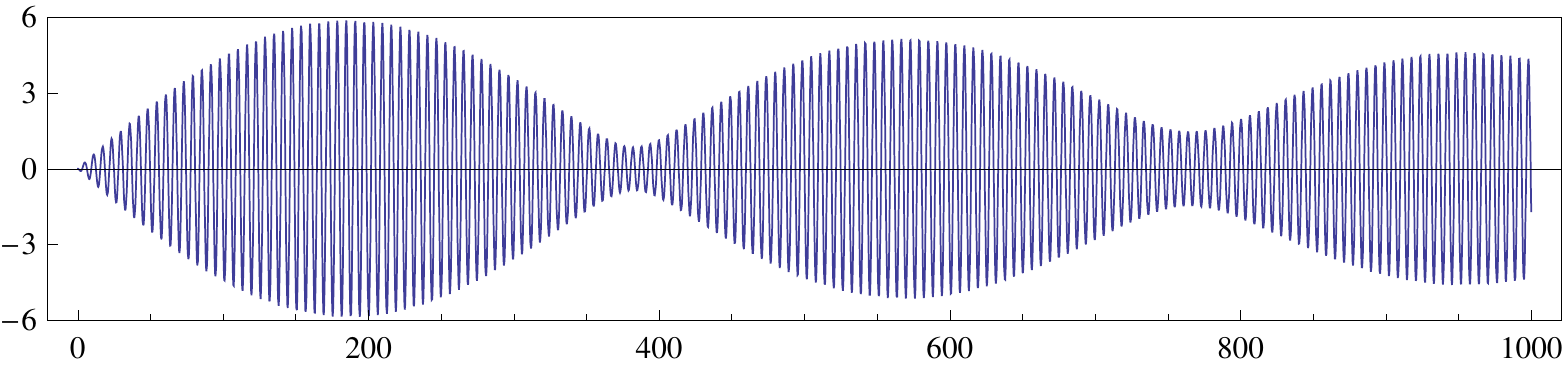} \\
		\includegraphics[width=0.99\textwidth]{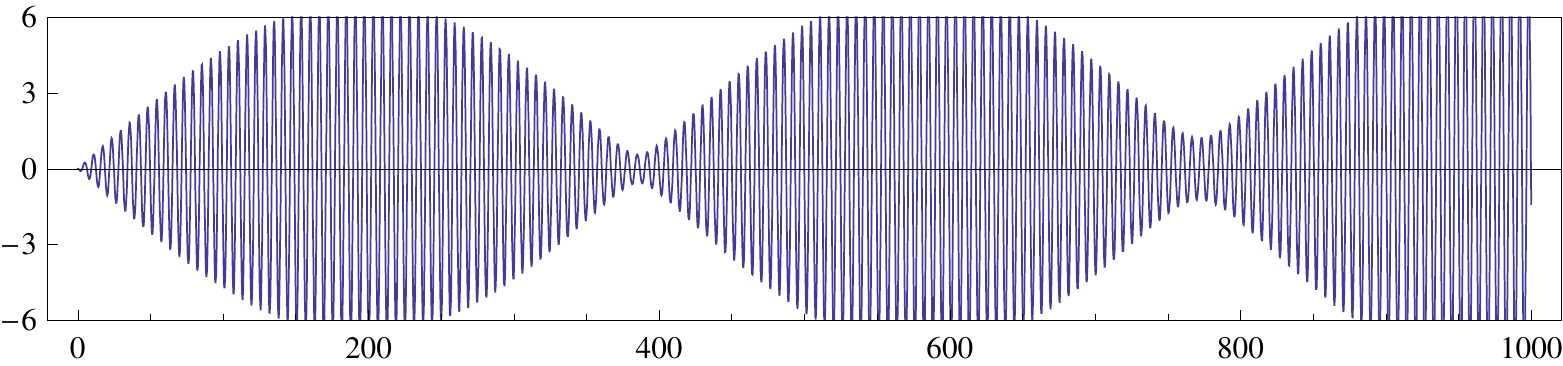} \\
		Time $t$
	\caption{Results for $q_1$ variable for the simple linear rotor example. 
	\emph{From top:} (1) exact solution, (2) implicit midpoint, geometric, 
	(3) Störmer--Verlet, geometric, (4) implicit Runge--Kutta, non-geometric,
	(5) explicit Runge-Kutta, non-geometric. All methods are second order accurate. 
	Notice the superior qualitative and
	quantitative behaviour of the geometric methods.}
	\label{fig:SimpleRotor_xplot}
\end{figure}

%

%



%


\section{Conclusions} 
\label{sec:conclusions}

A structural analysis of non-autonomous systems has been carried out using the framework
of Lie sub-algebras of the Lie algebra of vector fields. As a direct application,
backward error analysis results are obtained for this class of problems.
Numerical examples of a classical rotor dynamical problem show that the geometric
methods preserving the structure of the problem indeed are favourable over
non-geometric dito. 




\paragraph{Acknowledgement} 
\label{par:acknowledgement}
	\onSKF{
		The author is grateful to Dag Fritzson, Claus Führer and Gustaf Söderlind
		for helpful discussions.
		The author would also like to thank SKF for the support given.	
	}{
		The authors are grateful to Claus Führer and Gustaf Söderlind
		for fruitful discussions.
		The work is supported by SKF and by the Swedish Research Council
		under contract \textsf{VR-621-2006-5737}.
	}


\bibliographystyle{abbrv} 
\bibliography{../../Papers/References}


\end{document}